\documentclass[12pt]{amsart}

\usepackage{amsthm, amsmath, amscd, amsfonts}
\usepackage{cases}
\usepackage{nameref, hyperref, cleveref}
\usepackage{fancyhdr}
\usepackage{array}
\usepackage[normalem]{ulem}
\usepackage{caption}
\usepackage{subcaption}
\usepackage{graphicx}
\usepackage[usenames,dvipsnames]{xcolor}
\usepackage{tikz}
\usepackage[color=WildStrawberry]{todonotes}
\usepackage[letterpaper]{geometry}
\usepackage[parfill]{parskip} 
\usetikzlibrary{shapes}

\newcommand{\R}{{\mathbb{R}}}
\newcommand{\Z}{\mathbb{Z}}
\newcommand{\N}{\mathbb{N}}

\newcommand{\lcm}{\textup{lcm}}

\newcommand{\y}{\mathbf{y}}

\newcommand{\G}{\mathcal{G}}
\newcommand{\HH}{\mathcal{H}}

\newcommand{\setn}{ \left\{ \G^{(C_n,L)} \right\} }

\newtheorem{theorem}{Theorem}[section]
\newtheorem{corollary}[theorem]{Corollary}
\newtheorem{lemma}[theorem]{Lemma}
\newtheorem{proposition}[theorem]{Proposition}

\newtheorem{definition}[theorem]{Definition}

\newtheorem*{CRT}{Chinese Remainder Theorem}

\definecolor{mygray}{cmyk}{0.1875,0,0.375,0.8745}
\definecolor{myraspberry}{cmyk}{0,1,0.15,0.3}
\definecolor{mysunshine}{cmyk}{0,0.104,0.542,0.0157}
\definecolor{mypurple}{cmyk}{0.324,0.743,0,0.488}

\tikzstyle{rectanglevertex}=
	[rounded corners=3pt,
	line width=1pt,
	draw=mygray,
	fill=black!5, 
	font=\fontsize{11}{11}]

\tikzstyle{vertex}=
	[rounded corners=3pt,
	line width=1pt,
	draw=mygray,
	fill=black!5, 
	font=\fontsize{11}{11}]
\tikzstyle{edge}=
	[line width=1.5pt,
	mygray]
\tikzstyle{edgelabel}=
	[font=\fontsize{11}{11}]

\tikzstyle{dashededge}=[line width=1.5pt,dotted]
\tikzstyle{dashedarrow}=[dashed, line width=2pt, ->]

\title{Integer Generalized Splines on Cycles}
\author[Handschy, Melnick, and Reinders]{Madeline Handschy, Julie Melnick, and Stephanie Reinders}

\begin{document}

\maketitle


\begin{abstract}
	Let $G$ be a graph whose edges are labeled by positive integers. Label each vertex with an integer and suppose if two vertices are joined by an edge, the vertex labels are congruent to each other modulo the edge label. The set of vertex labels satisfying this condition is called a generalized spline. Gilbert, Polster, and Tymoczko recently defined generalized splines based on work on polynomial splines by Billera, Rose, Haas, Goresky-Kottwitz-Machperson, and many others. We focus on generalized splines on $n$-cycles. We construct a particularly nice basis for the module of splines on $n$-cycles. As an application, we construct generalized splines on star graphs, wheel graphs, and complete graphs.
\end{abstract}

\section{Introduction}

Let $G$ be a graph whose edges are labeled by positive integers. Label each vertex with an integer so that if two vertices are joined by an edge, the vertex labels are congruent to each other modulo the edge label. If each vertex label satisfies this condition, the set of vertex labels is called a generalized spline. In this paper, we often refer to generalized splines simply as splines.  Figure \ref{Fig:TriSystem} is an example of a spline on a triangle and the system of congruences that it represents. 

\begin{figure}
\begin{tikzpicture}
	
	\begin{scope}
		\pgfmathsetmacro{\r}{1.25}

		\draw[edge] (-90:\r)--(0:\r);
		\draw[edge] (0:\r)--(90:\r);
		\draw[edge] (90:\r)--(-90:\r);		
	
		\node[edgelabel] at (-45:\r) {$\ell_1$};
		\node[edgelabel] at (45:\r) {$\ell_2$};
		\node[edgelabel] at (180:\r /4) {$\ell_3$};
		
		\node[vertex] at (-90:\r) {$g_1$};
		\node[vertex] at (0:\r) {$g_2$};
		\node[vertex] at (90:\r) {$g_3$};
	\end{scope}

	\begin{scope}[xshift=1.75in]
		\node at (0,0) {
			$\begin{cases}
				g_1 \equiv g_2 \bmod \ell_1 \\
				g_2 \equiv g_3 \bmod \ell_2 \\
				g_3 \equiv g_1 \bmod \ell_3 \\
			\end{cases}$};
	\end{scope}
\end{tikzpicture}
\caption{}
\label{Fig:TriSystem}
\end{figure}

Mathematicians adopted the term spline from engineering. Before the invention of computer-aided design, ship builders would construct models using lead weights and splines, by which they meant thin strips of wood or metal. They placed the weights at key points and connected them with splines. The splines would bend, creating smooth curves.  In various areas of mathematics, a smooth curve is created by piecing together polynomials so that  at the point where two polynomials meet their derivatives are equal. Mathematicians chose the term spline to refer to the piecewise polynomial functions used to create smooth curves. This method of constructing curves is used in data interpolation \cite{Kunis07, Schumaker11}, in computer design software to create smooth surfaces \cite{Marsh05, Salomon06}, and partial differential equations to find numerical solutions \cite{Gupta12, Mittal12}. 

Splines also have a rich history in homological and commutative algebra, as well as geometry and topology (although not necessarily under the same name). Goresky, Kottwitz and MacPherson \cite{GKM}, Schenck \cite{Schenck12}, Payne \cite{Payne06},  and others considered splines from a geometric and topological perspective because splines represent equivariant cohomology rings. Billera \cite{Billera88, Billera91-2}, Haas \cite{Haas91}, Rose \cite{Rose95, Rose04},  and others studied the homological and algebraic properties of splines; their work is closely related to our research. Billera defined splines as piecewise polynomials on a polyhedron. Billera and Rose \cite{Billera91-2} used dual graphs of polyhedra and had definitions similar to ours. Haas \cite{Haas91} and others looked for bases of splines. 

Adapting Billera, Haas and Rose's notion of splines, Gilbert, Polster and Tymoczko developed the following definition of generalized splines \cite{Gilbert}. The ring of generalized splines $R_G$ of the pair $(G,\alpha)$ is defined by a ring $R$, a graph $G$, and a function $\alpha: E(G) \to \{\text{ideals in } \R \}$ such that
\[ R_G = \{ p \in R^{|V|} : \text{ for each edge } e=uv, \text{ the difference } p_u-p_v \in \alpha(e) \}. \]
Our work builds directly off of this. We work in the ring of integers. Ideals are always principle in the integers so we simply label each edge with the generator of the ideal.

Cycles form a particularly important family of graphs and are one of the ``building blocks'' of spline theory. Rose studied cycles and proved that they can be used to generate syzygies of spline modules \cite{Rose95, Rose04}. Gilbert, Polster, and Tymoczko showed that the ring of splines for an arbitrary graph can be decomposed as an instersection of the ring of splines of a spanning tree and the ring of splines for various cycles \cite{Gilbert}.

In this paper we prove that given any edge-labeled $n$-cycle particularly nice splines called flow-up classes exist. We explicitly identify the smallest leading element of each flow-up class and define the {\it smallest} flow-up classes. We go on to prove that the smallest flow-up classes form a basis for the module of splines over the integers. A flow-up basis is analogous to an upper-trianglular basis and is especially interesting to geometers and topologists. This type of basis shows up naturally in Schubert classes. It also occurs when considering certain kinds of group actions (Morse flow in topology or Bialynicki-Birula's study of algebraic torus actions). The main difference between standard flow-up classes (like Schubert classes) used by Knutson and Tao \cite{Knutson}, Goldin and Tolman \cite{Goldin09}, and Tymoczko \cite{Tymoczko} and our flow-up classes is that Schubert classes require the leading nonzero entry to satisfy a specific condition based on the graph. When working with integers, we can find flow-up classes that don't satisfy this condition.

Combinatorists and geometers have long been interested in explicit formulas for the entries of flow-up bases. In the case of Schubert classes, this formula is often called Billey's Formula \cite{Billey} or Anderson-Jantzen-Soergel's formula \cite{Anderson}. We give a closed formula for the leading entry of the smallest flow-up basis in Theorem 4.5.

Indetifying splines depends upon finding a solution to a system of congruences. For that reason, the Chinese Remainder Theorem is a key component of the majority of our proofs. In fact, we will see that the star graph is itself a representation of the Chinese Remainder Theorem. 

In Section 2 we introduce definitions and notation foundational to splines. We also introduce flow-up classes, as well as the formulation of  the Chinese Remainder Theorem that is most useful in our context. In Section 3 we address the particular case of splines on triangles. We prove that flow-up classes exist and that the smallest flow-up classes form a basis for the module of splines over the integers. We generalize these proofs for $n$-cycles in Section 4. In Section 5 we identify necessary and sufficient conditions for splines to exist on star graphs and then apply that information to construct splines on wheel graphs and complete graphs.

\section{Preliminaries}

In this section we introduce notation and definitions, including the definition of flow-up classes. We also give examples of splines. 

\subsection{The Chinese Remainder Theorem}
We recall the Chinese Remainder Theorem, which is a key piece of many of our proofs because it determines when a solution to a system of congruences exists. 
\begin{CRT}
\label{CRT}
	Choose $g_1,g_2,...,g_n \in \Z$ and  $\ell_1,\ell_2,...,\ell_n\in \N$. There exists an integer $x$ solving
	\[ x \equiv g_i \bmod \ell_i\]
for each $i$ in $1 \leq i \leq n$ if and only if $g_i \equiv g_j \bmod \gcd(\ell_i, \ell_j)$ for all $1 \leq i,j \leq n$.  All solutions $x$ are congruent modulo the $\lcm(\ell_1,\ell_2,...,\ell_n)$.
\end{CRT}
In this paper, the integers $\ell_i$ will correspond to edge-labels of a graph and the integers $g_i$ will correspond to vertex labels. We construct splines by labeling one vertex at a time. The Chinese Remainder Theorem comes into play when we attempt to find an integer that satisfies the edge conditions at a specific vertex. Many modern algebra textbooks use a special case of the Chinese Remainder Theorem where the $\ell_i$ are pairwise relatively prime. We use the stronger, more general statement because it gives us a tool to also work with edge-labels that are not relatively prime.

\subsection{Foundations of Splines}
We will give definitions and propositions that are used throughout this work. To start, we formally define an edge-labeled graph.
\begin{definition}[Edge-Labeled Graphs]
	Let $G$ be a graph with $k$ edges ordered $e_1, e_2, \dots, e_k$ and $n$ vertices ordered $v_1,...,v_n$.   Let $\ell_i$ be a positive integer label on edge $e_i$ and let $L=\{\ell_1,...,\ell_k\}$ be the set of edge labels. Then $(G,L)$ is an edge-labeled graph.  
\end{definition}

With this definition, we are now able to define splines.
\begin{definition}[Generalized Splines]
	A generalized spline on the edge-labeled graph $(G,L)$ is a vertex labeling satisfying the following: if two vertices are connected by an edge $e_i$ then the labels on the two vertices are equivalent modulo the label on the edge. We denote a generalized spline $\G^{(G,L)}=(g_1,...,g_n)$ where $g_i$ is the label on vertex $v_i$ for $1 \leq i \leq n$. The set of generalized splines is denoted $\left\{ \G^{(G,L)}\right\}$. 
\end{definition}

Figure \ref{TriSquareStar} gives examples of splines on a triangle, a square and a star graph. Sections 3 and 4 focus specifically on cycles and Section 5 focuses on star graphs, wheel graphs, and complete graphs.

\begin{figure}
\begin{tikzpicture}
	
	\begin{scope}
		\pgfmathsetmacro{\r}{1.25}

		\draw[edge] (-90:\r)--(0:\r);
		\draw[edge] (0:\r)--(90:\r);
		\draw[edge] (90:\r)--(-90:\r);		
	
		\node[edgelabel] at (-45:\r) {$2$};
		\node[edgelabel] at (45:\r) {$3$};
		\node[edgelabel] at (180:\r /4) {$5$};
		
		\node[vertex] at (-90:\r) {$0$};
		\node[vertex] at (0:\r) {$2$};
		\node[vertex] at (90:\r) {$5$};
	\end{scope}

	\begin{scope}[xshift=2in]
		\pgfmathsetmacro{\r}{1.25}
		
		\draw[edge] (-90:\r)--(0:\r);
		\draw[edge] (0:\r)--(90:\r);
		\draw[edge] (90:\r)--(180:\r);
		\draw[edge] (180:\r) -- (-90:\r);

		\node[edgelabel] at (-45:\r) {$5$};
		\node[edgelabel] at (45:\r) {$2$};
		\node[edgelabel] at (135:\r) {$4$};
		\node[edgelabel] at (225:\r) {$8$};
	
		\node[vertex] at (-90:\r) {$1$};
		\node[vertex] at (0:\r) {$11$};
		\node[vertex] at (90:\r) {$13$};
		\node[vertex] at (180:\r) {$17$};
	\end{scope}

	\begin{scope}[xshift=4in]
		\pgfmathsetmacro{\r}{1.25}
		
		\draw[edge] (-90:\r) -- (90:\r);
		\draw[edge] (0:\r) -- (180:\r);
		
		\node[edgelabel,right] at (-90:\r /2) {$3$};
		\node[edgelabel,above] at (0:\r /2) {$7$};
		\node[edgelabel,left] at (90:\r /2) {$5$ };
		\node[edgelabel,below] at (180:\r /2) {$6$ };

		\node[vertex] at (0,0) {$10$ };
		\node[vertex] at (-90:\r) {$7$};
		\node[vertex] at (0:\r) {$3$};
		\node[vertex] at (90:\r) {$5$};
		\node[vertex] at (180:\r) {$4$};
	\end{scope}

\end{tikzpicture}
\caption{}
\label{TriSquareStar}
\end{figure}

	We adopted a standard convention for numbering the vertices and edges of $n$-cycles. Given an $n$-cycle, the vertex $v_1$ is the lowest vertex. From $v_1$ we number the other vertices counterclockwise. Then we number the edges as follows: 
\begin{enumerate} 
\item edge $e_i$ is the edge between $v_i$ and $v_{i+1}$ for $1 \leq i \leq n-1$ \\
\item  edge $e_n$ is the edge between $v_1$ and $v_n$. \\
\end{enumerate}

Gilbert, Polster, and Tymoczko showed that the set of splines always form a module over the integers with the standard componentwise addition and multiplication for $n$-tuples \cite[Theorem 2.12]{Gilbert}. 

\begin{figure}
\begin{tikzpicture}
	
	\begin{scope}
		\pgfmathsetmacro{\r}{1.25}

		\draw[edge] (-90:\r)--(0:\r);
		\draw[edge] (0:\r)--(90:\r);
		\draw[edge] (90:\r)--(-90:\r);		
	
		\node[edgelabel] at (-45:\r) {$5$};
		\node[edgelabel] at (45:\r) {$4$};
		\node[edgelabel] at (180:\r /4) {$6$};
		
		\node[vertex] at (-90:\r) {$0$};
		\node[vertex] at (0:\r) {$0$};
		\node[vertex] at (90:\r) {$12$};
	\end{scope}

	\begin{scope}[xshift=2in]
		\pgfmathsetmacro{\r}{1.25}

		\draw[edge] (-90:\r)--(0:\r);
		\draw[edge] (0:\r)--(90:\r);
		\draw[edge] (90:\r)--(-90:\r);		
	
		\node[edgelabel] at (-45:\r) {$6$};
		\node[edgelabel] at (45:\r) {$10$};
		\node[edgelabel] at (180:\r /4) {$15$};
		
		\node[vertex] at (-90:\r) {$0$};
		\node[vertex] at (0:\r) {$0$};
		\node[vertex] at (90:\r) {$30$};
	\end{scope}

\end{tikzpicture}
\caption{The leading entry of flow-up classes need not be the product of the edges below.}
\label{CanonicalClass}
\end{figure}
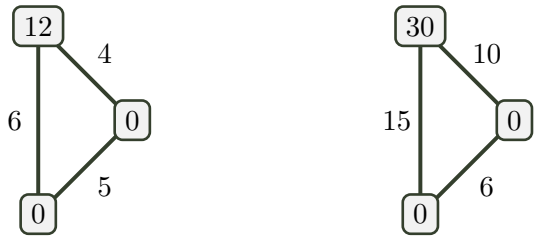

\begin{figure}
\begin{tikzpicture}
	
	\begin{scope}
		\pgfmathsetmacro{\r}{1.25}

		\draw[edge] (-90:\r)--(0:\r);
		\draw[edge] (0:\r)--(90:\r);
		\draw[edge] (90:\r)--(-90:\r);		
	
		\node[edgelabel] at (-45:\r) {$2$};
		\node[edgelabel] at (45:\r) {$3$};
		\node[edgelabel] at (180:\r /4) {$5$};
		
		\node[vertex] at (-90:\r) {$1$};
		\node[vertex] at (0:\r) {$1$};
		\node[vertex] at (90:\r) {$1$};
	\end{scope}

	\begin{scope}[xshift=2in]
		\pgfmathsetmacro{\r}{1.25}

		\draw[edge] (-90:\r)--(0:\r);
		\draw[edge] (0:\r)--(90:\r);
		\draw[edge] (90:\r)--(-90:\r);		
	
		\node[edgelabel] at (-45:\r) {$2$};
		\node[edgelabel] at (45:\r) {$3$};
		\node[edgelabel] at (180:\r /4) {$5$};
		
		\node[vertex] at (-90:\r) {$0$};
		\node[vertex] at (0:\r) {$2$};
		\node[vertex] at (90:\r) {$5$};
	\end{scope}

	\begin{scope}[xshift=4in]
		\pgfmathsetmacro{\r}{1.25}

		\draw[edge] (-90:\r)--(0:\r);
		\draw[edge] (0:\r)--(90:\r);
		\draw[edge] (90:\r)--(-90:\r);		
	
		\node[edgelabel] at (-45:\r) {$2$};
		\node[edgelabel] at (45:\r) {$3$};
		\node[edgelabel] at (180:\r /4) {$5$};
		
		\node[vertex] at (-90:\r) {$0$};
		\node[vertex] at (0:\r) {$0$};
		\node[vertex] at (90:\r) {$15$};
	\end{scope}

\end{tikzpicture}
\caption{Flow-up classes $\G_0,\G_1$ and $\G_2$ on a triangle}
\label{TriBasis}
\end{figure}
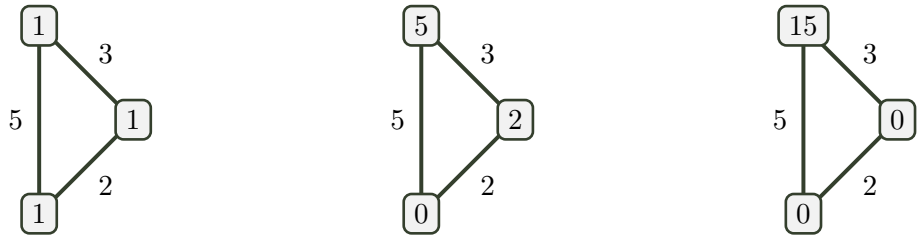

\subsection{Flow-Up Classes}

A flow-up class is a particularly nice type of spline. Given an edge-labeled graph $(C_n,L)$ a flow-up class $\G_k$ has $k$-leading zeroes. We will prove in Sections 3 and 4 that flow-up classes exist for each $0 \leq k < n$. These splines are particularly nice because they are linearly independent.

\begin{definition}[Flow-Up Classes]
	 Fix a cycle with edge labels $(C_n,L)$ and fix $k$ with $1\leq k<n$. A flow-up class $\G_k$ on $(C_n,L)$ is a spline with $k$ leading zeroes. The flow-up class $\G_0$ is a multiple of the trivial spline.
\end{definition}

Our concept of flow-up classes is an adaptation of  the canonical classes used by Tymoczko \cite{Tymoczko} and  Goldin and Tolman \cite{Goldin09} and the Schubert classes used by Knutson and Tao \cite{Knutson}. In their classes the leading entry on vertex $v_i$ is the product of the labels on the edges between $v_i$ and vertices lower than it. This is the case in our flow-up classes if the edge labels are relatively prime, but if the edge labels are not relatively prime we find leading entries that are not multiples of the labels on the edges below. Figure \ref{CanonicalClass} shows examples of this.

Figure \ref{TriBasis} shows the three flow-up classes $\G_0, \G_1$ and $\G_2$ on a triangle. By convention we build a flow-up class on an edge-labeled graph by starting with the bottom vertex and labeling the vertices counterclockwise. The $k$ lowest vertices are labeled zero. Then we label the remaining vertices in order.  Because each flow-up class has one more zero than its predecessor, the nonzero labels in essence ``flow up" the graph - hence the name.

At times it is more convenient to write the flow-up class $\G_k=(g_1,...,g_n)$ as a vector:
\[ \G_k=\begin{pmatrix}
	g_n \\
	\vdots \\
	g_1 \\
\end{pmatrix}\] 
Notice that in vector form we write the vertex labels from bottom to top instead of the standard top to bottom. We chose this notation so the elements of the vector would correspond to the height of vertices in the graph.

We were interested in not only finding flow-up classes but finding the smallest flow-up classes. For both the triangle and the general $n$-cycle we were able to explicitly identify the smallest leading element of a flow-up class. For the remaining elements, we were able to prove that a smallest element must exist without characterizing it explicitly. We give the formal definition of the smallest flow-up classes here. In Sections 3 and 4 we reveal the smallest leading entry and prove that the smallest flow-up classes exist.

\begin{definition}[Smallest Flow-Up Class]
	Fix a cycle with edge labels $(C_n,L)$. The smallest flow-up class $\G_k=(0,...,0,g_{k+1},...,g_n)$ on $(C_n,L)$ is the flow-up class whose nonzero entries are positive and if $\G_k'=(0,...,0,g_{k+1}',...,g_n')$ is another flow-up class with positive entries then $g_i' \geq g_i$ for all entries. By convention we consider \\ $\G_0=(1,...,1)$ the smallest flow-up class $\G_0$.
\end{definition}

Two types of flow-up classes are easily found on any edge-labeled $n$-cycle: trivial splines and splines with one nonzero label. If we label each vertex 1 on an edge-labeled $n$-cycle, we always obtain a spline, which we call a trivial spline. A trivial spline is a special case of the flow-up class $\G_0$. 

\begin{proposition}[Existence of Trivial Splines]
\label{TrivialSplines}
	Fix a cycle with edge labels $(C_n,L)$. The smallest flow-up class  on $(C_n,L)$ is $\G_0=(1,...,1)$. Moreover, any multiple of $\G_0$  is also a generalized spline. We call these splines trivial splines. 
\end{proposition}

\begin{proof}
	The reader can easily verify that $\G_0=(1,...,1)$ is a spline on $(C_n,L)$. Suppose we can find a smaller flow-up class $\G_0'=(g_1,...,g_n)$. The leading entry $g_1$ must be a positive entry so it must be greater than or equal to 1. We know we can construct a spline with the first vertex label equal to 1 so $g_1=1$ is the smallest entry for the first vertex. The second entry $g_2$ must be the smallest nonnegative integer that in particular satisfies the congruence
\[ 1 \equiv g_2 \bmod \ell_1. \]
Zero is not congruent to 1 for any $\ell_1$ so $g_2 \geq 1$. Again, we know that we can construct a spline with $g_2=1$ so $g_2=1$ is the smallest entry for the second vertex. Assume that the smallest nonnegative integer label for $v_k$ (for $2 \leq k <n$) is 1. Then the $g_{k+1}$ must be congruent to 1 modulo the edge label $\ell_k$. Zero is not an option so $g_{k+1} \geq 1$. We can construct a spline with all vertex labels 1 so $g_{k+1}=1$ is the smallest entry for the $k+1$ vertex. Thus $\G_0=(1,...,1)$ is the smallest flow-up class $\G_0$ on $(C_n,L)$.
\end{proof}

Splines with only one nonzero label also exist on any edge-labeled $n$-cycle. If all other vertices besides $v_i$ are labeled 0, we can label $v_i$ with any multiple of the labels on the two adjoining edges. The following proposition states this in greater detail. 
\begin{proposition}[Existence of Flow-Up Class with One Nonzero Label]
\label{OneNonzeroLabel}
	Fix a cycle with edge labels $(C_n,L)$. The flow-up class $\G_{n-1}=(0,...,0,g_n)$ is a spline on $(C_n,L)$ whenever the vertex label $g_n$ is a multiple of $\lcm(\ell_{n-1},\ell_n)$. If $g_n=\lcm(\ell_{n-1},\ell_n)$ then $\G_{n-1}$ is the smallest flow-up class.
\end{proposition}

\begin{proof}
	We can see that $\G_{n-1}=(0,...,0,\lcm(\ell_{n-1}, \ell_n))$  is a spline because it satisfies the system of congruences
	\[ \begin{cases}
		0 \equiv 0 \bmod \ell_1 \\
		\vdots \\
		x \equiv 0 \bmod \ell_{n-1} \\
		x \equiv 0 \bmod \ell_n. \\
	\end{cases} \]
	
	Now we need to show that $\lcm(\ell_{n-1},\ell_n)$ is the smallest positive integer that satisfies the system of congruences. By definition, $x$ must be a multiple of $\ell_{n-1}$ and $\l_n$. It follows that $x=\lcm(\ell_{n-1},\ell_n)$ is the smallest positive solution to the system. Thus $\G_{n-1}$ is the smallest flow-up class. 
\end{proof}

Establishing that other flow-up classes $\G_k$ for $2 \leq k < n-1$ exist is slightly more difficult on triangles and significantly harder on $n$-cycles. The proof for triangles is in the next section and the proof for $n$-cycles is in Section 4. In both of these sections, after showing that the flow-up classes exist, we will prove the smallest flow-up classes also exist and form a basis for the module of splines over the integers.

\section{Triangles}

In this section we address splines on triangles. Our primary goal is to prove that the smallest flow-up classes form a basis for the module of splines over the integers. In order to do this, we need to establish that the smallest flow-up classes exist on any edge-labeled triangle. We split this into several parts. First we prove that flow-up classes $\G_0,\G_1$ and $\G_2$ exist. Next we find the smallest leading element of the flow-up classes. Then we establish the existence of the smallest flow-up classes. We wrap up the section by proving that the smallest flow-up classes form a basis.

The flow-up classes $\G_0$ and $\G_2$ exist  by Proposition \ref{TrivialSplines} and Proposition \ref{OneNonzeroLabel}  respectively, so we only need to prove that the flow-up class $\G_1$ exists. 

\begin{theorem}[Existence of Flow-Up Class $\G_1$ on Triangles]
\label{Tri:FlowUpClass}
	Fix a cycle with edge labels $(C_3,L)$. The flow-up class $\G_1=(0, g_2, g_3)$ exists on $(C_3,L)$ if and only if $g_2$ is a multiple of $\lcm(\ell_1,\gcd(\ell_2,\ell_3))$. 
\end{theorem}

\begin{proof}
		A flow-up class  $\G_1=(0,g_2,g_3)$ exists if and only if  we are able to find integers $g_2$ and $g_3$ that satisfy the following system of congruences:
\[ \begin{cases}
g_2 \equiv 0 \bmod \ell_1 \\
g_3 \equiv g_2 \bmod \ell_2 \\
g_3 \equiv 0 \bmod \ell_3 \\
\end{cases} \]		
By the Chinese Remainder Theorem, there exists an integer $g_3$ satisfying the second and third congruences if and only if $g_2 \equiv 0 \bmod \gcd(\ell_2,\ell_3)$. By a similar application of the Chinese Remainder Theorem,  we see that $g_2$ simultaneously satisfies this congruence and the first congruence in the system if and only if $g_2$ is a multiple of $\lcm(\ell_1, \gcd(\ell_2,\ell_3))$. If we let $g_2$ be any multiple of $\lcm(\ell_1,\gcd(\ell_2,\ell_3))$, then we can find an integer $g_3$ that satisfies the system. Consequently,  $\G_1=(0,g_2,g_3)$ is a spline.
\end{proof}

\begin{theorem}[Smallest Leading Element of $\G_1$ on Triangles]
\label{Tri:SmallestLeadingElement}
Fix a cycle with edge labels $(C_3,L)$. Let $\G_1=(0,g_2,g_3)$ be a flow-up class on $(C_3,L)$. There is a flow-up class $\G_1$ for which the leading element $g_2=\lcm(\ell_1,\gcd(\ell_2,\ell_3))$ and this is the smallest flow-up class $\G_1$.
\end{theorem}

\begin{proof}
Let $\G_1=(0,g_2,g_3)$ be a flow-up class on $(C_3,L)$. In the proof of Theorem \ref{Tri:FlowUpClass} we showed that $g_2$ must be a multiple of $\lcm(\ell_1,\gcd(\ell_2,\ell_3))$. Every choice of $g_2$ that satifies this condition has at least one $g_3$ by the Chinese Remainder Theorem. Thus $g_2=\lcm(\ell_1,\gcd(\ell_2,\ell_3))$ is the smallest positive value that $g_2$ can be.

\end{proof}

\begin{theorem}[Existence of Smallest Flow-Up Class $\G_1$ on Triangles]
\label{Tri:SmallestFlowUpClass}
Fix a cycle with edge labels $(C_3,L)$.The smallest flow up class $\G_1$ exists on $(C_3,L)$.
\end{theorem}

\begin{proof}	
Let $\G_1=(0,g_2,g_3)$ be a flow-up class on $(C_3,L)$ and let $g_2$ be the smallest leading element. By Theorem \ref{Tri:FlowUpClass} we know $g_2=\lcm(\ell_1,\gcd(\ell_2,\ell_3))$. Consider the coset of all integers $g_3$ that would satisfy the edge conditions at $v_3$ making $\G_1$ a spline. Well-Ordering tells us that a smallest positive $g_3$ exists. If we choose that integer as $g_3$ then $\G_1$ is the smallest flow-up class. 
\end{proof}

Note that we didn't need to find $g_3$ explicitly to prove $\G_1$ is the smallest flow-up class. In order to fit the definition of the smallest flow-up class we only needed to show that a smallest $g_3$ exists. 

Now that we demonstrated that the smallest flow-up classes $\G_0,\G_1$ and $\G_2$ exist, we will prove they form a basis for the module of splines. 

\begin{theorem}[Basis for Triangles]
\label{Tri:Basis}
	Fix a cycle with edge labels $(C_3,L)$. Let $\G_0,\G_1$ and $\G_2$ be the smallest flow-up classes on $(C_3,L)$. Then the linear span  of $\{ \G_0, \G_1, \G_2\}$ is a basis for the module of splines  over the integers.
\end{theorem}

\begin{proof}
	The splines are linearly independent because each has a different number of leading zeroes. 
	
	The harder step will be to show that every spline is in the span of these flow-up classes. Let $\y=(y_1,y_2,y_3)$ be a spline on $(C_3,L)$. Define $\y'$ as
		\[ \y' = \y-y_1\G_0 = \begin{pmatrix} y_3 \\ y_2 \\ y_1 \end{pmatrix} - y_1\begin{pmatrix} 1 \\ 1 \\ 1 \end{pmatrix} = \begin{pmatrix} y_3-y_1 \\ y_2-y_1 \\ 0 \end{pmatrix} \]
		Because $\y'$ is a linear combination of splines $\y$ and $\G_0$ and the set of splines is a module, the vector $\y'$ is also a spline. The leading element $y_2-y_1$ of $\y'$ is an integer multiple of $\lcm(\ell_1,\gcd(\ell_2,\ell_3))$ by Theorem \ref{Tri:FlowUpClass}. The smallest flow-up class $\G_1=(0,g_2,g_3)$ has the leading element $g_2=\lcm(\ell_1,\gcd(\ell_2,\ell_3))$ by Theorem \ref{Tri:SmallestLeadingElement}. This means we can find an $s \in \Z$, with $y_2-y_1=s \cdot g_2$.
		
		Now subtract $s \G_1$ from $\y'$ and call the difference $\y''$. 
		\[ \y'' = \y'-s\G_1 = 
			\begin{pmatrix} 
				y_3-y_1 \\
				y_2-y_1 \\ 
				0 \\
			 \end{pmatrix} 
			- s
			\begin{pmatrix} 
				g_{3} \\ 
				g_{2} \\ 
				0 \\
			\end{pmatrix} 
			=
			 \begin{pmatrix} 
				y_3-y_1-s\cdot g_{3} \\ 
				y_2-y_1-s \cdot g_{2} \\ 
				0 \\ 
			\end{pmatrix} 
			=
			\begin{pmatrix} 
				y_3-y_1-s\cdot g_{3} \\ 
				0 \\
				0 \\ 
			\end{pmatrix} \]
		This is a spline because the set of splines forms a module. Thus $y_3-y_1-s\cdot g_3$ satisfies the edge conditions at vertex $v_3$ so it must be a multiple of $\ell_2$ and $\ell_3$. For some integer $t$ we can write $y_3-y_1 - s\cdot g_{3} = t\cdot \lcm(\ell_2,\ell_3)$.  It follows that
		\[ \y'' - t \G_3 = 
			\begin{pmatrix} 
				y_3-y_1 - s \cdot g_{3} \\
				0 \\
				0 \\
			\end{pmatrix} 
			- t 
			\begin{pmatrix} 
				\lcm(\ell_2,\ell_3) \\
				0 \\
				0 \\ 
			\end{pmatrix} 
			= 
			\begin{pmatrix} 
				0 \\ 
				0 \\ 
				0 \\ 
			\end{pmatrix}. \]
		Then
		\[ \y = y_1 \G_0 + s \G_1 + t \G_2 \]
		for $y_1, s, t \in \Z$.  Therefore $\y$ is an integer linear combination of $\G_0, \G_1, \G_2$. Thus $ \{\G_0, \G_1, \G_2\}$ forms a basis over the integers for the splines on $(C_3,L)$.
\end{proof}

\section{$n$-Cycles}

The concepts of the previous section can be generalized from triangles to $n$-cycles. As we did for triangles, we will prove that smallest flow-up classes exist on $n$-cycles and form a basis for the module of splines over the integers. The proofs for $n$-cycles are more involved so we introduce two lemmas to simplify the proofs. 

The first lemma is that if we have a flow-up class with two leading zeroes on an $n$-cycle, we can contract the first edge to get a flow-up class with one leading zero on an $n-1$ cycle. The second lemma states that this process can be reversed. 

\begin{lemma}[Edge Contraction]
\label{EdgeContraction}
Fix a cycle with edge labels $(C_n,L)$. Let $(0,0,g_3,...,g_n)$ be a spline on $(C_n,L)$. If  we contract the edge labeled $\ell_1$ the vertex labels $(0,g_3,...,g_n)$ will be a spline on $(C_{n-1},L')$ where $L'=(\ell_2,...,\ell_n)$.
\end{lemma}

\begin{proof}
Let $(0,0,g_3,...,g_n)$ be a spline on $(C_n,L)$. Then the following system holds
		\[ \begin{cases}
			0 \equiv 0 \bmod \ell_1 \\
			0 \equiv g_3 \bmod \ell_2 \\
			\vdots \\
			g_n \equiv 0 \bmod \ell_n. 
		\end{cases} \]
		If we contract edge $\ell_1$ then we remove the congruence $0 \equiv 0 \bmod \ell_1$ from the system. The rest of the system still holds so we have a spline on $(C_{n-1},L')$ with $L'=(\ell_2,...,\ell_n)$.
\end{proof}

\begin{lemma}[Adding an Edge]
\label{EdgeAddition}
Fix a cycle with edge labels $(C_{n-1},L')$ where \newline $L'=(\ell_2,...,\ell_{n})$. Let $(0,g_3,...,g_n)$ be a spline on $(C_{n-1},L')$. 
Consider the cycle $(C_n,L)$ formed by inserting a vertex $v_0$ labeled zero and an edge labeled $\ell_1$ into the cycle so that the edge labeled  $\ell_{n}$ connects $v_n$  to $v_0$ and the edge labeled $\ell_1$ connects $v_1$ to $v_0$.  Then $(0,0,g_3,...,g_n)$ is a spline on $(C_n,L)$.
\end{lemma}

\begin{proof}
Let $(0,g_2,...,g_{n-1})$ be a spline on $(C_{n-1},L')$ with $L'=(\ell_1,...,\ell_{n-1})$. Then the following system holds
		\[ \begin{cases}
			0 \equiv g_2 \bmod \ell_1 \\
			\vdots \\
			g_n \equiv 0 \bmod \ell_n. 
		\end{cases} \]
The congruence $0 \equiv 0 \bmod \ell_0$ is met for any $\ell_0 \in \N$. If we add this as the first congruence in the system above, the system will still be satisfied. This new system represents a spline $(0,0,g_2,...,g_{n-1})$ on $(C_n,L)$ where $L=(\ell_0,...,\ell_{n-1})$.

\end{proof}

The next theorem states that flow-up classes exist on all edge-labeled $n$-cycles.

\begin{theorem}[Existence of Flow-Up Classes]
\label{ncycles:FlowUpClasses}
	Fix a cycle with edge labels $(C_n,L)$. Let $n\geq 3$ and $1 \leq k<n$. There exists a flow-up class $\G_k$ on $(C_n,L)$.
\end{theorem}

\begin{proof} 
	First we consider the case where $k=1$ and then we will cover the case where  $2 \leq k<n$. 

Let $k=1$ and $n\geq 3$. The flow-up class $\G_1$ exists on $(C_n,L)$ if and only if the following system of congruences is satisfied by some $g_2,g_3,...,g_n$
\[ \begin{cases} 
	0 \equiv g_2 \bmod \ell_1 \\
	g_2 \equiv g_3 \bmod \ell_2 \\
	\vdots \\
	g_{n-1} \equiv g_n \bmod \ell_{n-1} \\
	g_n \equiv 0 \bmod \ell_n \\
\end{cases} \]
The first congruence $0 \equiv g_2 \bmod \ell_1$ is satisfied if $g_2 = p_1 \ell_1$ for $p_1 \in \Z$. If we choose $g_2=p_1 \ell_1$ then the second congruence is satisfied if $g_3 = p_1 \ell_1 + p_2 \ell_2$ for $p_2 \in \Z$. We can see a pattern: the $i$-th congruence is satified if $g_{i+1} = p_1 \ell_1 + \cdots p_i \ell_i$ for $p_1,...,p_i \in \Z$. Choose $g_{i+1} = p_1 \ell_1 + \cdots p_i \ell_i$ for each $i$ in $1 \leq i \leq n-2$. Then the first $n-2$ congruences are met. In particular, 
\[ g_{n-1} = p_1 \ell_1 + \cdots + p_{n-2} \ell_{n-2}. \]
Factor out $\gcd(\ell_1,...,\ell_{n-2})$ from the right side.
\[ g_5 = \gcd(\ell_1,...,\ell_{n-2})\left(\dfrac{p_1 \ell_1}{\gcd(\ell_1,...,\ell_{n-2})} + \cdots + \dfrac{p_{n-2} \ell_{n-2}}{\gcd(\ell_1,...,\ell_{n-2})}\right) \]
We know $g_{n-1}$ is a multiple of $\gcd(\ell_1,...,\ell_{n-2})$, which is the same as saying
\begin{equation} g_{n-1} \equiv 0 \bmod \gcd(\ell_1,...,\ell_{n-2}). \end{equation}
So far we have satisfied the first $n-2$ congruences. We still need to find a solution for the last two congruences.  

A solution $g_n$ exists if and only if 
\[ g_{n-1} \equiv 0 \bmod \gcd(\ell_{n-1},\ell_n) \]
by the Chinese Remainder Theorem. Choose $g_{n-1}$ to be a multiple of \newline $\lcm(\gcd(\ell_1,...,\ell_{n-2}) , \gcd(\ell_{n-1},\ell_n))$.

Then $g_{n-1}$ satisfies the congruences in equations (1) and (2). Because $g_{n-1}$ satisfies the first of these congruences a solution $g_n$ exists to the last two congruences in the original system.

We constructed $g_2,...,g_{n-2}$ to meet the first $n-3$ congruences and we constructed $g_{n-1}$ to satisfy  both the $(n-2)$th congruence and the congruence \newline $g_{n-1} \equiv 0 \bmod \gcd(\ell_{n-1},\ell_n)$. Because $g_{n-1}$ satisfies the latter congruence, the Chinese Remainder Theorem guarantees we can find $g_n$ that satisfies the $(n-1)$th and $n$-th congruences. Thus a solution exists to the entire system, in other words a flow-up class $\G_1$ exists on $(C_n,L)$.

Next consider the case where $n\geq 3$ and $2 \leq k <n$. We proved in Proposition \ref{OneNonzeroLabel} that the flow-up class $\G_2$ exists for $n=3$. This is our base case. Assume the flow-up classes $\G_2,...,\G_{n-1}$ exist for $(C_{n-1},L')$ where $L'=(\ell_2,...,\ell_n)$. Additionally, $\G_1$ also exists on this $n$-cycle by the earlier part of this proof. Add a leading zero to each of the flow-up classes $\G_1,...,\G_{n-1}$ on $(C_{n-1},L')$. Lemma \ref{EdgeAddition} assures us these new vectors are splines on $(C_{n},L)$. Moreover the additional leading zero makes these flow-up classes $\G_2,...,\G_n$. Thus we conclude by induction that for all $n \geq 3$ and $2 \leq k <n$ the flow-up class $\G_k$ exists on $(C_n,L)$.

\end{proof}

In Section 3 we explicitly identified the smallest leading element of a flow-up class on triangles. Similarly we can explicitly identify the smallest leading element of each flow-up class on $n$-cycles. First we will address the flow-up class $\G_1$. Then we build upon that to find the smallest leading element of flow-up classes in general.

\begin{theorem}[Smallest Leading Element of $\G_1$]
\label{ncycles:SmallestLeadingElement}
Fix a cycle with edge labels $(C_n,L)$. Fix $n\geq 3$. Let $\G_1=(0,g_2,...,g_n)$ be a flow-up class on $(C_n,L)$. The leading element $g_2$ is a multiple of $\lcm(\ell_1,\gcd(\ell_2,...,\ell_n))$ and $g_2=\lcm(\ell_1,\gcd(\ell_2,...,\ell_n))$ is the smallest positive value such that $\G_1$ is a spline.
\end{theorem}

\begin{proof}
We will prove this by induction. The base case $n=3$ was proven in Theorem 3.2.

Let $n>3$. For any edge-labeled cycle $(C_n,L)$, the flow-up class $\G_1=(0,g_2,...,g_n)$ exists on $(C_n,L)$ by Theorem 4.3. Assume the induction hypothesis that $g_2$ is a multiple of $\lcm(\ell_1,\gcd(\ell_2,...,\ell_n))$ and $g_2=\lcm(\ell_1,\gcd(\ell_2,...,\ell_n))$ is the smallest positive value such that $\G_1$ is a spline. Fix a cycle with edge labels $(C_{n+1},L')$ with $L'=\{\ell_1,...,\ell_{n+1}\}$. The flow-up class $\G_1'=(0,g_2',...,g_{n+1}')$ exists on $(C_{n+1},L')$ by Theorem 4.3. This means the following system of congruences is satisfied:
\[ \begin{cases}
g_2' \equiv 0 \bmod \ell_1 \\
g_3' \equiv g_2' \bmod \ell_2 \\
\vdots \\
g_n' \equiv g_{n-1}' \bmod \ell_{n-1} \\
g_{n+1}' \equiv g_n' \bmod \ell_n \\
g_{n+1}' \equiv 0 \bmod \ell_{n+1}. \\
\end{cases} \]
Because the last two congruences are satisfied,  $g_n' \equiv 0 \bmod \gcd(\ell_n,\ell_{n+1})$ by the Chinese Remainder Theorem. Consider this congruence and the first $n-1$ congruences in the system above. Together these congruences form a flow-up class $\G_1''=(0,g_2',...,g_n')$ on $(C_n,L'')$ with edge labels $L''=(\ell_1,...,\gcd(\ell_n,\ell_{n+1}))$. By the induction hypothesis, for some $s\in \Z$ 
\[ g_2' = s \lcm(\ell_1,\gcd(\ell_2,...,\gcd(\ell_n,\ell_{n+1}))) = s \lcm (\ell_1,\gcd(\ell_2,...,\ell_{n+1})) \]
and $g_2' = \lcm(\ell_1,\gcd(\ell_1,...,\ell_{n+1}))$ is the smallest positive value that satisfies the edge conditions at the second vertex.
\[ \begin{cases}
g_2' \equiv 0 \bmod \ell_1 \\
g_3' \equiv g_2' \bmod \ell_2 \\
\vdots \\
g_n' \equiv g_{n-1}' \bmod \ell_{n-1} \\
g_n' \equiv 0 \bmod \gcd(\ell_n,\ell_{n+1}) \\
\end{cases} \]
This system is a flow-up class $\G_1''=(0,g_2',...,g_n')$ on $(C_n,L')$ with \\ $L'=(\ell_1,...,\ell_{n-1},\gcd(\ell_n,\ell_{n+1}))$. Moreover by induction there exists a flow-up class $\G_1'$ on $(C_n,L'')$ with $g_2' =\lcm(\ell_1,\gcd(\ell_1,...,\ell_{n+1}))$.  This class also satisfies \newline $g_n' = 0 \bmod \gcd (\ell_n,\ell_{n+1})$ so there is at least one solution $g_{n+1}'$ to the system for $\G_1'$ on $(C_{n+1},L')$ with the  same values for $g_2', g_3',..., g_n'$.  This is a flow-up class $\G_1'$ satisfying $g_2' = \lcm(\ell_1,\gcd(\ell_1,...,\ell_{n+1}))$ as desired. 
\end{proof}

\begin{theorem}[Smallest Leading Element of $\G_k$] 
\label{ncycles:SmallestLeadingElementGk}
Fix a cycle with edge labels $(C_n,L)$. Fix $n\geq 3$ and $k$ with $3 \leq k <n$. Let $\G_{k-1}=(0,...,0,g_k...,g_n)$ be a flow-up class on $(C_n,L)$. The leading element $g_k$ is a multiple of $\lcm(\ell_{k-1},\gcd(\ell_k,...,\ell_n))$ and $g_k=\lcm(\ell_{k-1},\gcd(\ell_k,...,\ell_n))$ is the smallest positive value that satisfies the $v_k$ edge conditions.
\end{theorem}

\begin{proof}
Let $\G_k=(0,...,0,g_{k+1},...,g_n)$ for $3 \leq k <n$ be a spline on $(C_n,L)$. Contract the first $k-1$ edges. The resulting graph is $(C_{n-k-1},L')$ with $L'=(\ell_k,...,\ell_n)$ and the resulting vector is $(0,g_{k+1},...,g_n)$. By multiple applications of Lemma \ref{EdgeContraction}, this vector is a spline. In other words, the following system of congruences is satisfied.
\[ \begin{cases} 
g_{k+1} \equiv 0 \bmod \ell_k \\
\vdots \\
g_n \equiv g_{n-1} \bmod \ell_{n-1} \\
g_n \equiv 0 \bmod \ell_n \\ 
\end{cases} \]
This system represents a $\G_1$ spline on $(C_{n-k-1},L')$. The leading element $g_{k+1}$ is a multiple of $\lcm(\ell_k,\gcd(\ell_{k+1},...,\ell_n))$ and $g_{k+1}=\lcm(\ell_k,\gcd(\ell_{k+1},...,\ell_n))$ is the smallest positive value that satisfies the $v_2$ edge conditions on $(C_{n-k-1},L')$. The 
the original spline $\G_k=(0,...,0,g_{k+1},...,g_n)$ satisfies the edge conditions on $(C_n,L)$ if and only if $\G_1=(0,g_{k+1},...,g_n)$ satisfies the edge conditions on $(C_{n-k-1},L')$ so the smallest leading element of $\G_1$ is also the smallest leading element of $\G_k$. Thus the leading element $g_{k+1}$ on $(C_n,L)$ must be a multiple of $\lcm(\ell_k,\gcd(\ell_{k+1},...,\ell_n))$ and $g_{k+1}=\lcm(\ell_k,\gcd(\ell_{k+1},...,\ell_n))$ is the smallest positive value that satisfies the $v_{k+1}$ edge conditions on $(C_n,L)$.
\end{proof}
 
\begin{theorem}[Existence of Smallest Flow-Up Classes]
\label{ncycles:SmallestFlowUpClasses}
Fix a cycle with edge labels $(C_n,L)$. Fix $n\geq 3$ and $k$ where $3 \leq k <n$. There exists a smallest flow-up class \newline $\G_k=(0,...,0,g_{k+1},...,g_n)$  on $(C_n,L)$.
\end{theorem}

\begin{proof}
Start with an edge-labeled $n$-cycle $(C_n,L)$ with $n\geq 3$. Let $1 \leq k <n$ and let $\HH_k=(0,...,0,h_{k,k+1},...,h_{k,n})$ be the flow-up class $\G_k$ on $(C_n,L)$ with the smallest leading entry. Theorems 4.3 and 4.5 guarantee that such flow-up classes exists. 

Consider the linear combination of flow-up classes with smallest leading entries
\[ \HH = \HH_k - c_{k+1} \HH_{k+1}  - \cdots - c_{n-1}\HH_{n-1} \]
for all $c_{k+1},...,c_{n-1} \in \Z$. For each $i$ where $1<i<n$  consider in turn the $(k+i)$th vertex label, which is 
\[ h_{k,k+i} - c_{k+1} h_{k+1,k+i} - \cdots - c_{k+i} h_{k+i,k+i+1} \]
By Well-Ordering one vertex label is the smallest nonnegative integer. Pick that integer for the $(k+i)$th vertex label and fix the coefficient $c_{k+1}$ that give that vertex label. Then the leading entry of $\HH$ is the smallest positive integer and all remaining entries are the smallest nonnegative integers such that $\HH$ is a spline. Thus $\HH$ is  the smallest flow-up class $\G_k$ by definition. 


\end{proof}

The smallest flow-up classes form a basis for $n$-cycles. We state this in the following theorem.

\begin{theorem}[Basis for $n$-Cycles]
\label{ncycles:Basis}
	The smallest flow-up classes $\G_0, \G_1, \dots, \G_{n-1}$ form a basis over the integers for the set of splines $\setn$.
\end{theorem}

\begin{proof} 
		The flow-up classes $\G_0,...,\G_{n-1}$ are linearly independent because each has a different number of leading zeroes.  
	
		Next we need to show that any spline in the set of splines on $(C_n,L)$ can be written as an integer linear combination of the $\G_0, \G_1, \dots, \G_{n-1}$.  

	The smallest flow-up classes $\G_0, \G_1, \G_2$  on $(C_3,L)$ form a basis over the integers for the set of splines on $(C_3,L)$ by Theorem \ref{Tri:Basis}. In particular, this means that any spline on $(C_3,L)$ can be written as a linear combination of the smallest flow-up classes. 
		
	Assume that any spline on $(C_n,L)$ can be written as a linear combination of the smallest flow-up classes $\G_0,...,\G_{n-1}$. 

	Let $\y=(y_1,...,y_{n+1})$ be a spline on $(C_{n+1},L)$. Then
	\[ \y' = \y - y_1 \G_0 = \y-y_1 \begin{pmatrix}
			1 \\
			\vdots \\
			1 \\
			1 \\
		\end{pmatrix}
		=
		\begin{pmatrix}
			y_{n+1}-y_1 \\
			\vdots \\
			y_2 - y_1 \\
			0 \\
		\end{pmatrix} =
		\begin{pmatrix}
			y_{n-1}' \\
			\vdots \\
			y_2' \\
			0 \\
		\end{pmatrix} \]

		is also a spline on $(C_{n+1},L)$ because the set of splines is a module.  Since 
\[ \y = y_1\G_0 + \y' \]
so we have written part of $\y$ as a linear combination of one flow-up class. Continuing on, we have the smallest flow-up class $\G_1=(0,g_2,...,g_{n+1})$ with \newline $g_2 = \lcm(\ell_1,\gcd(\ell_2,...,\ell_{n+1}))$.  
		Then $y_2' = s g_2$ for some $s\in \Z$ by Theorem \ref{ncycles:SmallestLeadingElement} Then 
		\[ \y'' = \y' - s \G_1 =
			\begin{pmatrix}
				y_{n+1}' \\
				\vdots \\
				y_3' \\
				y_2' \\
				0 \\
			\end{pmatrix}
		- s
			\begin{pmatrix}
				g_{n+1} \\
				\vdots \\
				g_3 \\
				g_2\\
				0 \\
			\end{pmatrix}
		=
			\begin{pmatrix}
				y_{n+1}' - s g_{n+1} \\
				\vdots \\
				y_3' - s g_3 \\
				0\\
				0 \\
			\end{pmatrix}
		=
			\begin{pmatrix}
				y_{n+1}'' \\
				\vdots \\
				y_3'' \\
				0\\
				0 \\
			\end{pmatrix} \]
		
		is also a spline on $(C_{n+1},L)$. 

We have now written part of $\y$ as a linear combination of two flow-up classes and the spline $\y''$.
\[ \y=y_1 \G_0 + s \G_1 + \y'' \]
We need to write spline $\y''$ as a linear combination of flow-up classes. This step involves a trick. We will define a helper function $f$ and leverage Lemma \ref{EdgeContraction} to use the smallest flow-up classes on $(C_n,L')$, which exist by our induction hypothesis, to construct smallest flow-up classes on $(C_{n+1},L)$.

		Define $(C_n,L')$ with edge labels $L'=(\ell_2,..., \ell_{n+1})$. Then  $(0, y_3'',..., y_{n+1}'')$ is a spline on $(C_n,{L'})$ by Lemma \ref{EdgeContraction}.  By our induction hypothesis,  this new spline  $(0, y_3'',..., y_{n+1}'')$ can be written as a linear combination of the set of smallest flow-up classes $\G_0',...,\G_{n-1}'$ on $(C_n,L')$. In other words, for $c_0,...,c_{n-1} \in \Z$, we can write
\[ \begin{pmatrix}
	y_{n+1}'' \\
	\vdots \\
	y_3'' \\
	0 \\
\end{pmatrix}
= c_0\G_0' + c_1 \G_1' + \cdots + c_{n-1}\G_{n-1}'. \]
The flow-up classes $\G_1',...,\G_{n-1}'$ have zero as their first entry, so $c_0=0$ and 
\[ \begin{pmatrix}
	y_{n+1}'' \\
	\vdots \\
	y_3'' \\
	0 \\
\end{pmatrix}
= c_1 \G_1' + c_2 \G_2' + \cdots  + c_{n-1} \G_{n-1}'. \]

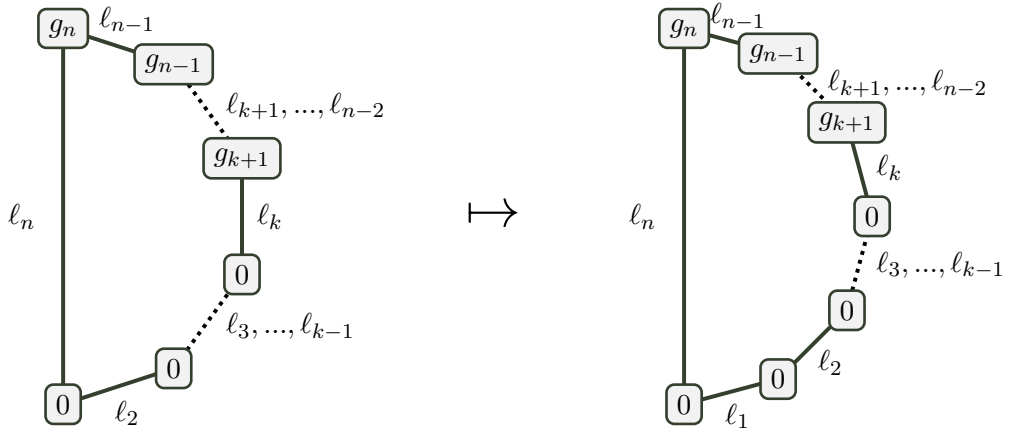
\begin{figure}
\begin{tikzpicture}

\begin{scope}
	\pgfmathsetmacro{\r}{2.5}
	\pgfmathsetmacro{\ro}{2.75}
	\pgfmathsetmacro{\edge}{36}

	\node (a) at ({-90 + \edge*0}:\r) {};
	\node (b) at ({-90 + \edge*1}:\r) {};
	\node (c) at ({-90 + \edge*2}:\r) {};
	\node (d) at ({-90 + \edge*3}:\r) {};
	\node (e) at ({-90 + \edge*4}:\r) {};
	\node (f) at ({-90 + \edge*5}:\r) {};

	\draw[edge] (a)--(b);
	\draw[dashededge] (b)--(c);
	\draw[edge] (c)--(d);
	\draw[dashededge] (d)--(e);
	\draw[edge] (e)--(f);
	\draw[edge] (f)--(a);

	\node[edgelabel] at ({-72 + \edge*0}:\ro) {$\ell_2$};
	\node[edgelabel,right] at ({-72 + \edge*1}:\r) {$\ell_3,...,\ell_{k-1}$};
	\node[edgelabel] at ({-72 + \edge*2}:\ro) {$\ell_k$};
	\node[edgelabel,right] at ({-72 + \edge*3}:\r) {$\ell_{k+1},...,\ell_{n-2}$};
	\node[edgelabel] at ({-72 + \edge*4}:\ro) {$\ell_{n-1}$};
	\node[edgelabel] at (180:\ro /5) {$\ell_n$};

	\node[rectanglevertex] at (a) {$0$};
	\node[rectanglevertex] at (b) {$0$};
	\node[rectanglevertex] at (c) {$0$};
	\node[rectanglevertex] at (d) {$g_{k+1}$};
	\node[rectanglevertex] at (e) {$g_{n-1}$};
	\node[rectanglevertex] at (f) {$g_n$};
\end{scope}

\begin{scope}[xshift=2.25in]
	\node[font=\fontsize{22}{22}] at (0,0) {$\mapsto$};
\end{scope}

\begin{scope}[xshift=3.25in]
	\pgfmathsetmacro{\r}{2.5}
	\pgfmathsetmacro{\ro}{2.75}
	\pgfmathsetmacro{\edge}{30}

	\node (a) at ({-90 + \edge*0}:\r) {};
	\node (b) at ({-90 + \edge*1}:\r) {};
	\node (c) at ({-90 + \edge*2}:\r) {};
	\node (d) at ({-90 + \edge*3}:\r) {};
	\node (e) at ({-90 + \edge*4}:\r) {};
	\node (f) at ({-90 + \edge*5}:\r) {};
	\node (g) at ({-90 + \edge*6}:\r) {};

	\draw[edge] (a)--(b);
	\draw[edge] (b)--(c);
	\draw[dashededge] (c)--(d);
	\draw[edge] (d)--(e);
	\draw[dashededge] (e)--(f);
	\draw[edge] (f)--(g);
	\draw[edge] (g)--(a);

	\node[edgelabel] at ({-75 + \edge*0}:\ro) {$\ell_1$};
	\node[edgelabel] at ({-75 + \edge*1}:\ro) {$\ell_2$};
	\node[edgelabel,right] at ({-75 + \edge*2}:\r) {$\ell_3,...,\ell_{k-1}$};
	\node[edgelabel,right] at ({-75 + \edge*3}:\r) {$\ell_k$};
	\node[edgelabel,right] at ({-75 + \edge*4}:\r) {$\ell_{k+1},...,\ell_{n-2}$};
	\node[edgelabel] at ({-75 + \edge*5}:\ro) {$\ell_{n-1}$};
	\node[edgelabel] at (180:\ro /5) {$\ell_n$};

	\node[rectanglevertex] at (a) {$0$};
	\node[rectanglevertex] at (b) {$0$};
	\node[rectanglevertex] at (c) {$0$};
	\node[rectanglevertex] at (d) {$0$};
	\node[rectanglevertex] at (e) {$g_{k+1}$};
	\node[rectanglevertex] at (f) {$g_{n-1}$};
	\node[rectanglevertex] at (g) {$g_n$};
\end{scope}

\end{tikzpicture}
\caption{The function $f$ maps $\G_{k-1}$ on $(C_{n-1},L')$ to $\G_k$ on $(C_n,L)$}
\end{figure}

Define a helper function $f$ that takes a flow-up class $\G_k'$ (for $1 \leq k <n-1$) on $(C_n,L')$ and maps it to a flow-up class $\G_{k+1}$ on $(C_{n+1},L)$ by adding a leading zero. We are able to define such a function because if $\G_k'$ is a spline on $(C_n,L')$ then adding a leading zero makes it a flow-up class $\G_{k+1}$ on $(C_{n+1},L)$ by Lemma \ref{EdgeAddition}. 

We want to prove $f$ sends the smallest flow-up class $\G_k'$ on $(C_n,L')$ to the smallest flow-up class $\G_{k+1}$ on $(C_{n+1},L)$. Start with the smallest flow-up class $\G_k'$ on $(C_n,L')$. Then $f(\G_k')=\G_{k+1}$ on $(C_{n+1},L)$. Because $\G_k'$ and $\G_{k+1}$ represent essentially the same system of congruences, $\G_{k+1}$ must be the smallest flow-up class.

Now return to our spline $\y''$ on $(C_n,L')$. We have
\[ \begin{pmatrix}
	y_{n+1}'' \\
	\vdots \\
	y_3'' \\
	0 \\
\end{pmatrix}
= c_1\G_1' + c_2 \G_2' +\cdots  + c_{n-1}\G_{n-1}'. \]
By applying the function $f$ to each side of the equation we produce the flow-up class
\[ \begin{pmatrix}
	y_{n+1}'' \\
	\vdots \\
	y_3'' \\
	0 \\
	0 \\
\end{pmatrix}
= c_1 \G_2 + c_2 \G_3 + \cdots  + c_{n-1}\G_n. \]
The function $f$ sends the smallest flow-up class $\G_k'$ to $\G_{k+1}$ so $\G_2,...,\G_n$ are the smallest flow-up classes on $(C_{n+1},L)$. Now we add this expression of $\y''$ as a linear combination of $\G_2,...,\G_n$ to our original expression for $\y$. 
\[ \y=y_1 \G_0 + s \G_1 + c_1 \G_2 + c_2 \G_3 + \cdots + c_{n-1} \G_n \]
We conclude that $\y$ can be written as a linear combination of the smallest flow-up classes. Thus by induction, for any $n\geq 3$ any spline on $(C_n,L)$ is in the span of the smallest flow-up classes.

We have shown that the smallest flow-up classes are linearly independent and form a spanning set for the set of splines. Thus they form a basis for the module of splines on $(C_n,L)$. 
\end{proof}

\section{Star Graphs, Wheel Graphs and Complete Graphs}

In this section we use our knowledge of splines on cycles to identify splines on wheel and complete graphs.  We first prove  a necessary and sufficient condition for a collection of integers to form a spline on a star graph, which follows directly from the Chinese Remainder Theorem. Then, instead of addressing wheels graphs and complete graphs directly, we view them as combinations of cycles and star graphs.

Thoughout this section we use $C_n$ to denote cycle graphs on $n$ vertices, $W_n$ to denote wheel graphs on $n$ vertices, $K_n$ to denote complete graphs on $n$ vertices, and $S_n$ to denote star graphs with $n$ edges (and $n+1$ vertices).

Recall our standard convention for numbering the vertices of an $n$-cycle: we number the lowest vertex on $C_n$ as $v_1$ and number the other vertices counterclockwise. We number the vertices of a star graph $S_n$ in a similar manner: number the lowest leaf $v_1$ and number the remain leaves as $v_2,...,v_n$ going counterclockwise. Number the central vertex $v_{n+1}$.

\subsection{Star Graphs}
Consider the system of congruences in the Chinese Remainder Theorem: for $g_1,...,g_k\in \Z$ and $\ell_1,...,\ell_k \in \N$, 
\[ \begin{cases} 
	x \equiv g_1 \bmod \ell_1 \\
	\vdots \\
	x \equiv g_k \bmod \ell_k \\
\end{cases} \]
The graph representation of this system is a star graph. Thus the Chinese Remainder Theorem gives a necessary and sufficient condition for a collection of integers to form a spline on a star graph with $k$ edges. We present this formally in the next lemma.

\begin{lemma}[Existence of Splines on Star Graphs] 
\label{StarGraphs}
A collectin of integers form a spline on a star graph $S_n$ if and only if the labels on the leaves are congruent to each other modulo the greatest common divisor of the edge labels. 
\end{lemma}

\begin{proof}
Start with a star graph $S_n$ and label each edge so that $\ell_i$ is the label on the edge between vertices $v_i$ and $v_{n+1}$ for $1\leq i \leq n$. Label each vertex $v_i$ as $g_i$ with $g_i \in \Z$ for $1\leq i \leq n+1$. The set of vertex labels $g_1,...,g_{n+1}$ is a spline on $S_n$ if and only if the following system of congruences is satisfied:
\[ \begin{cases}
g_{n+1} \equiv g_i \bmod \ell_i  & \text{for } 1 \leq i \leq n \\
\end{cases} \]
By the Chinese Remainder Theorem, a solution $g_{n+1}$ exists if and only if $g_1,...,g_{n}$ are congruent to each other modulo the greatest common divisor of the edge labels $\ell_1,...,\ell_{n}$.
\end{proof}

Wheel graphs can be constructed by starting with a cycle $C_n$ and joining a single vertex $v$. In other words $W_{n+1} = C_n + v$. We can also think of this process as joining  a cycle graph $C_n$ and a star graph $S_n$ by merging the cycle's vertex $v_i$ and the star's leaf $v_i$ for $1\leq i \leq n$, as shown in Figure \ref{WheelGraph}.This is nontandard graph theory notation, but it allows us to apply what we know about splines on star graphs. 

Number the edges of $W_{n+1}$ by numbering the edges of its component parts as follows:
\begin{enumerate}
\item  Number the edges of $C_n$ so that $e_i$ is the edge between $v_i$ and $v_{i+1}$ for \newline $1\leq i \leq n-1$ and $e_n$ is  the edge between $v_1$ and $v_n$. 
\item Number the edges of $S_n$ so that $e_{n+i}$ is the edge between $v_i$ and $v_{n+1}$ for $1\leq i \leq n$. 
\end{enumerate}
Our convention for labeling the edges of $W_{n+1}$ is to label edge $e_i$ as $\ell_i$ for $1\leq i \leq n$ and $\ell_{n+i}$ for $n+1 \leq i \leq 2n$.

 Using what we know about splines on cycles and star graphs we have the following corollary for wheel graphs.

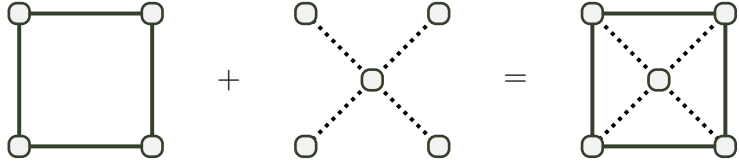
\begin{figure}
\begin{tikzpicture}
\begin{scope}
	\pgfmathsetmacro{\r}{1.25}
	
	\draw[edge] (-45:\r)--(45:\r);
	\draw[edge] (45:\r)--(135:\r);
	\draw[edge] (135:\r)--(225:\r);
	\draw[edge] (225:\r) -- (-45:\r);
	
	\node[vertex] at (-45:\r) { };
	\node[vertex] at (45:\r) { };
	\node[vertex] at (135:\r) { };
	\node[vertex] at (225:\r) { };
\end{scope}

\node[xshift=0.75in] at (0,0) {$+$};

\begin{scope}[xshift=1.5in]
	\pgfmathsetmacro{\r}{1.25}
	
	\draw[dashededge] (225:\r)--(45:\r);
	\draw[dashededge] (135:\r)--(315:\r);
	
	\node[vertex] at (0,0) { };
	\node[vertex] at (-45:\r) { };
	\node[vertex] at (45:\r) { };
	\node[vertex] at (135:\r) { };
	\node[vertex] at (225:\r) { };
\end{scope}

\node[xshift=2.25in] at (0,0) {$=$};

\begin{scope}[xshift=3in]
	\pgfmathsetmacro{\r}{1.25}
	
	\draw[edge] (-45:\r)--(45:\r);
	\draw[edge] (45:\r)--(135:\r);
	\draw[edge] (135:\r)--(225:\r);
	\draw[edge] (225:\r) -- (-45:\r);
	
	\draw[dashededge] (225:\r) -- (45:\r);
	\draw[dashededge] (-45:\r) -- (135:\r);
	
	\node[vertex] at (0,0) { };
	\node[vertex] at (-45:\r) { };
	\node[vertex] at (45:\r) { };
	\node[vertex] at (135:\r) { };
	\node[vertex] at (225:\r) { };
\end{scope}
\end{tikzpicture}
\caption{Constructing a wheel graph: $C_4 + S_4 = W_5$}
\label{WheelGraph}
\end{figure}

\begin{corollary}[Existence of Splines on Wheel Graphs]
\label{WheelGraphs}
	Construct an edge-labled wheel graph $W_{n+1}$ from a cycle graph $C_n$ and a star graph $S_n$ with $n \geq 3$.  Label each vertex $v_i$ by $g_i$ with $g_i \in \Z$. The set of vertex labels is a spline on $W_{n+1}$ if and only if the following conditions are satisfied:
\begin{enumerate}
\item  $g_i \equiv g_{i+1} \bmod \ell_i$ for $1 \leq i \leq n-1$ and $g_1 \equiv g_n \bmod \ell_n$. 

(I.e. the vertex labels $g_1,...,g_n$ satisfy the edge conditions of $C_n$)
\item $g_1,...,g_{n+1}$ are congruent to each other modulo $\gcd(\ell_{n+1},...,\ell_{2n})$. 

(I.e. the vertex labels $g_1,...,g_{n+1}$ satisfy the edge conditions of $S_n$)
\end{enumerate}
\end{corollary}

\begin{proof}
Start with an edge labeled wheel graph $(W_{n+1},L)$. In order for a spline to exist on $W_{n+1}$ the system of congruences represented by the cycle and the system represented by the star must both be satisfied. We constructed in Section 4 a basis for splines on any edge-labeled cycle. A collection of integers forms a spline on the star if and only if the labels on the vertices of degree three are congruent to each other modulo the greatest common divisor of the star's edge labels by Lemma 5.1. This means a spline exists on $W_{n+1}$ if and only if the labels on the vertices of degree three simultaneously satisfy the edge conditions of the cycle and are congruent to each other modulo the greatest common divisor of the star's edge labels. 
\end{proof}

We can build complete graphs from a cycle and a series of star graphs similarly to how we built wheel graphs. Start with a cycle on three vertices. Add a star graph with three edges so that each of the leaves merge with a distinct vertex of the cycle. This produces a complete graph on four vertices. Continue in this way, as shown in Figure \ref{CompleteGraph}. We describe this by the equation
\[ K_n = C_3 + \sum_{i=3}^{n-1} S_i \]
In standard graph theory notation, this operation could be written
\[ K_n = C_3 + v_1 + v_2+ \cdots +v_n \]
but we use the nonstandard notation to exploit our knowledge of star graphs. 

We number the vertices of a complete graph by numbering the lowest vertex of $C_3$ as $v_1$ and the remaining two vertices of $C_3$ as $v_2$ and $v_3$ going counterclockwise. Each time we add a star $S_i$ we add a single vertex to the complete graph. Number this new vertex $v_i$. 

Our convention for labeling the edges of a complete graph is as follows:
\begin{enumerate}
\item Label the edges of $C_3$ so that $\ell_1$ is the label on the edge between $v_1$ and $v_2$, $\ell_2$ is the label on the edge between $v_2$ and $v_3$, and $\ell_3$ is the label on the edge between $v_1$ and $v_3$
\item For each star graph $S_i$ label the edge between $v_j$ and $v_{i+1}$ for $1\leq j \leq i$ as $\ell_{i,j}$. 
\end{enumerate}

The following corollary stems from our proofs about splines on cycles and star graphs.

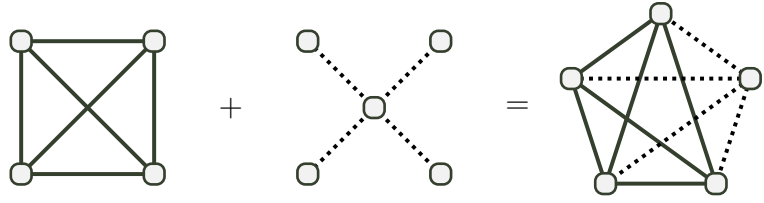
\begin{figure}
\begin{tikzpicture}
\begin{scope}
	\pgfmathsetmacro{\r}{1.25}
	
	\draw[edge] (-45:\r)--(45:\r);
	\draw[edge] (45:\r)--(135:\r);
	\draw[edge] (135:\r)--(225:\r);
	\draw[edge] (225:\r) -- (-45:\r);
	\draw[edge] (225:\r) -- (45:\r);
	\draw[edge] (-45:\r) -- (135:\r);
	
	\node[vertex] at (-45:\r) { };
	\node[vertex] at (45:\r) { };
	\node[vertex] at (135:\r) { };
	\node[vertex] at (225:\r) { };
\end{scope}

\node[xshift=0.75in] at (0,0) {$+$};

\begin{scope}[xshift=1.5in]
	\pgfmathsetmacro{\r}{1.25}
	
	\draw[dashededge] (225:\r)--(45:\r);
	\draw[dashededge] (135:\r)--(315:\r);
	
	\node[vertex] at (0,0) { };
	\node[vertex] at (-45:\r) { };
	\node[vertex] at (45:\r) { };
	\node[vertex] at (135:\r) { };
	\node[vertex] at (225:\r) { };
\end{scope}

\node[xshift=2.25in] at (0,0) {$=$};

\begin{scope}[xshift=3in]
	\pgfmathsetmacro{\r}{1.25}
	
	\draw[edge] (90:\r) -- (162:\r);
	\draw[edge] (162:\r) -- (234:\r);
	\draw[edge] (234:\r) -- (306:\r);
	\draw[edge] (90:\r) -- (234:\r);
	\draw[edge] (90:\r) -- (306:\r);	
	\draw[edge] (162:\r) -- (306:\r);
	
	\draw[dashededge] (18:\r) -- (90:\r);
	\draw[dashededge] (18:\r) -- (162:\r);
	\draw[dashededge] (18:\r) -- (234:\r);
	\draw[dashededge] (18:\r) -- (306:\r);
	
	\node[vertex] at (18:\r) { };
	\node[vertex] at (90:\r) { };
	\node[vertex] at (162:\r) { };
	\node[vertex] at (234:\r) { };
	\node[vertex] at (306:\r) { };
\end{scope}
\end{tikzpicture}
\caption{Constructing a complete graph: $K_4 + S_4 = K_5$}
\label{CompleteGraph}
\end{figure}

\begin{corollary}[Existence of Splines on Complete Graphs] 
\label{CompleteGraphs}
Construct an edge-labeled complete graph $K_{n+1}$ from a cycle graph $C_3$ and a series of star graphs $S_3,...,S_n$ with $n\geq 3$. Label each vertex $v_i$ by $g_i$ with $g_i \in \Z$. The set of vertex labels is a spline on $K_{n+1}$ if and only if the following conditions are satisfied:
\begin{enumerate}
\item the vertex labels $g_1,g_2,g_3$ satisfy the edge conditions of $C_3$.
\item for each star $S_i$ the labels on the leaves are congruent to each other modulo the greatest common divisor of the edge labels of $S_i$. I.e. $g_1,...,g_i$ are congruent to each other modulo $\gcd(\ell_1,...,\ell_i)$.
\end{enumerate}
\end{corollary} 

\begin{proof} The proof follows quickly from our construction of the complete graph from a cycle and stars, the results in Section 3 about splines on $C_3$  and Lemma 5.1 for star graphs. 

Let $K_{n+1}$ be an edge-labeled complete graph. A spline exists on $K_{n+1}$ if and only if the systems of congruences represented by the cycle and each of the star graphs are satisfied. In Section 3  we constructed a basis for any edge-labeled cycle $C_3$. For each star $S_3,...,S_n$ we know a collection of integers froms a spline if and only if the labels on the leaves are congruent to each other modulo the greatest common divisor of the star's edge labels by Lemma 5.1.
\end{proof}

\section{Acknowledgements}
First and foremost, we would like to acknowledge the invaluable help and guidance of Julianna Tymoczko. We also offer our thanks to Elizabeth Drellich, Ruth Haas, and Lauren Rose for useful feedback. This work was partially funded by the National Science Foundation (DMS grant 1143716) and the Center for Women in Mathematics at Smith College. The 3rd author was partially supported by Smith College SURF funding.


\end{document}